\documentclass{amsart}

\usepackage{amsmath,amsthm,amssymb}
\usepackage{mathtools}
\usepackage{tikz}

\usetikzlibrary{shapes.geometric}

\theoremstyle{plain}
\newtheorem{thm}{Theorem}[section]
\newtheorem{lem}[thm]{Lemma}
\newtheorem{cor}[thm]{Corollary}
\newtheorem{prop}[thm]{Proposition}

\newtheorem{thma}{Theorem}

\theoremstyle{definition}
\newtheorem{defn}[thm]{Definition}
\newtheorem{rem}[thm]{Remark}

\newtheorem{example}[thm]{Example}

\newcommand{\C}{\mathbb{C}}
\newcommand{\sym}{\textsc{Sym}}
\newcommand{\sgn}{{\operatorfont sgn}}
\newcommand{\partition}{\textsc{Part}}

\newcommand{\ba}{\mathbf{a}}

\newcommand{\bz}{\mathbf{z}}
\newcommand{\by}{\mathbf{y}}
\newcommand{\bj}{\mathbf{J}}
\newcommand{\lt}{\textsc{lt}}

\newcounter{joncomments}

\newcounter{michaelcomments}


\begin{document}
	
\title[Critical Points and Critical Values]{Critical points, critical values, 
and a determinant identity for complex polynomials}
\author{Michael Dougherty} 
    \email{michael.dougherty@colby.edu}
    \address{Department of Mathematics and Statistics, Colby College,
    Waterville, ME 04901}
\author{Jon McCammond}
    \email{jon.mccammond@math.ucsb.edu}
    \address{Department of Mathematics, UC Santa Barbara, 
    Santa Barbara, CA 93106} 

\begin{abstract}
    Given any $n$-tuple of complex numbers, one can canonically define 
    a polynomial of degree $n+1$ that has the entries of this $n$-tuple 
    as its critical points.  In 2002, Beardon, Carne, and Ng studied a 
    map $\theta\colon \C^n\to \C^n$ which outputs the critical values 
    of the canonical polynomial constructed from the input, and they 
    proved that this map is onto. Along the way, they showed that $\theta$ 
    is a local homeomorphism whenever the entries of the input are 
    distinct and nonzero, and, implicitly, they produced a polynomial 
    expression for the Jacobian determinant of $\theta$. In this 
    article we extend and generalize both the local homeomorphism result 
    and the elegant determinant identity to analogous situations where 
    the critical points occur with multiplicities. This involves 
    stratifying $\C^n$ according to which coordinates are equal and 
    generalizing $\theta$ to a similar map $\C^\ell \to \C^\ell$ where 
    $\ell$ is the number of distinct critical points. The more 
    complicated determinant identity that we establish is closely 
    connected to the multinomial identity known as Dyson's conjecture.
\end{abstract}

\maketitle

%%%%%%%%%%%%%%%%
% Introduction %
%%%%%%%%%%%%%%%%

Let $p(z) \in \C[z]$ be a polynomial and let $z_0\in \C$ be a complex 
number.  If $p'(z_0)=0$, then $z_0$ is a \emph{critical point of $p$} 
and its image $p(z_0)$ is a \emph{critical value of $p$}. For any 
$n$-tuple $\bz = (z_1,\ldots,z_n) \in \mathbb{C}^n$, one can define a 
canonical polynomial $p = p_\bz$ that has the entries of $\bz$ as its 
critical points:
\[p(z) = p_\bz(z) = \int_0^z (w-z_1) \cdots (w-z_n)\ dw\]
Beardon, Carne, and Ng used this polynomial to define a map 
$\theta\colon\C^n \to \C^n$ which sends $\bz=(z_1,\ldots,z_n)$ to 
$\theta(\bz) 
= (\theta_1(\bz),\ldots,\theta_n(\bz)) 
= (p(z_1),\ldots,p(z_n))$ 
with $p=p_\bz$. Note that $\theta$ sends the critical points of $p$ to 
the critical values of $p$. Beardon, Carne and Ng show that every 
$n$-tuple of complex numbers arises as the critical values of some 
polynomial by proving that this map $\theta$ is surjective \cite{beardon02}.

The Jacobian matrix $\bj = \bj(\bz)$ of the map $\theta$ at $\bz$ is 
the $n\times n$ matrix with $(i,j)$-entry given by 
$(\bj)_{ij} = \frac{\partial}{\partial z_i} \theta_j(\bz) 
= \frac{\partial}{\partial z_i} p(z_j)$. 
As part of their proof, Beardon, Carne, and Ng show that $\bj$ is 
invertible whenever $\bz$ has distinct nonzero entries. Our first 
theorem provides an alternate proof via an explicit computation for the 
determinant of $\bj$. For any positive integer $n$, we use $[n]$ to 
denote the set $\{1,\ldots,n\}$.

\begin{thma}\label{thma:beardon-det}
    Let $\bz = (z_1,\ldots,z_n)\in \mathbb{C}^n$ and let 
    $\bj = \bj(\bz)$ be the Jacobian matrix defined above. The 
    Jacobian determinant factors as follows:
    \[
        \det \bj = \frac{1}{n!} \left( \prod_{j\in [n]} (-z_j)\right) 
        \left( \prod_{\substack{j,k\in [n]\\ j\neq k}} (z_k-z_j)\right)  
    \]
	Thus $\bj$ is invertible if and only if
    $z_1,\ldots, z_n$ are distinct and nonzero.
\end{thma}

We prove a generalization of this determinant identity by 
focusing on the case where the critical points have specified 
multiplicities. Let $\bz = (z_1,\ldots,z_m) \in \mathbb{C}^m$
and let $\ba = (a_1,\ldots,a_m)$ be an $m$-tuple 
of positive integers with $n=a_1+\cdots +a_m$.
If we define the polynomial $p_\ba = p_{\bz,\ba}$ by the formula 
\[
    p_\ba(z) = p_{\bz,\ba}(z) 
    = \int_0^z (w-z_1)^{a_1} \cdots (w-z_m)^{a_m}\ dw,
\] 
then $p_\ba$ is the unique polynomial of degree $n+1$ which has
monic derivative, no constant term, and has $z_i$ as a critical point
with multiplicity $a_i$ for each $i\in [m]$.  Then the 
map $\theta_\ba\colon\C^m \to \C^m$, defined by sending each
$\bz=(z_1,\ldots,z_m) \in \mathbb{C}^m$ to $\theta_{\ba}(\bz) 
= (\theta_{\ba,1}(\bz),\ldots,\theta_{\ba,m}(\bz)) 
= (p_\ba(z_1),\ldots,p_\ba(z_m))$, takes the critical points 
of $p_\ba$ to the critical values of $p_\ba$ with the 
appropriate multiplicities. Our main theorem describes a 
factorization of the determinant of the $m\times m$ Jacobian 
matrix $\bj_\ba$, defined by 
$(\bj_\ba)_{ij} = \frac{\partial}{\partial z_i} \theta_{\ba,j}(\bz) 
= \frac{\partial}{\partial z_i} p_\ba(z_j)$.

\begin{thma}\label{thma:main-specific}
    Let $\bz = (z_1,\ldots,z_m) \in \mathbb{C}^m$, let 
    $\ba = (a_1,\ldots,a_m)$ be an $m$-tuple of positive integers 
    with $a_1+\cdots + a_m = n$, and let $\bj_\ba = \bj_\ba(\bz)$ 
    be the Jacobian matrix defined above. The Jacobian determinant 
    factors as follows:
	\[
	    \det \bj_\ba = \frac{1}{\binom{n}{a_1,\ldots,a_m}}
	    \left(\prod_{j\in [m]} (-z_j)^{a_j}\right)
	    \left(\prod_{\substack{j,k\in [m]\\j\neq k}}
	    (z_k-z_j)^{a_j}\right)
	\]
	Thus $\bj_\ba$ is invertible if and only if $z_1,\ldots,z_m$ 
	are distinct and nonzero.
\end{thma}

By the Inverse Function Theorem, the determinant result proven 
in \cite{beardon02} implies that the map 
$\theta\colon\mathbb{C}^n \to \mathbb{C}^n$ is a local 
homeomorphism at points with distinct nonzero entries. 
Theorem~\ref{thma:main-specific} allows us to extend this local 
homeomorphism result to points with nondistinct but nonzero 
entries. To do so, we stratify the points in $\mathbb{C}^n$ with 
nonzero entries according to which entries are equal, where the 
strata are in bijection with the partitions of $[n]$. In this 
setting, Theorem~\ref{thma:main-specific} demonstrates that while 
$\theta$ may or may not be a local homeomorphism at a generic point 
with nonzero entries, the restriction of $\theta$ to the 
corresponding stratum \emph{is} a local homeomorphism.

\begin{thma}\label{thma:main-verbose}
    Let $\lambda$ be a partition of $[n]$ and let $\C^\lambda$ be 
    the subspace of $\C^n$ consisting of all points
    $\bz = (z_1,\ldots,z_n)$ where $z_i = z_j$ if and only if
    $i$ and $j$ belong to the same block. Then $\theta$ restricts
    to a map $\theta_\lambda \colon \C^\lambda \to \overline{\C^\lambda}$,
    where $\overline{\C^\lambda}$ is the closure of $\C^\lambda$, and
    $\theta_\lambda$ is a local homeomorphism at $\bz \in \C^\lambda$
    if the entries of $\bz$ are nonzero.
\end{thma}

These theorems have direct connections to the braid group as the 
fundamental group of the space of complex polynomials with distinct 
roots. The details will be given in upcoming work by the authors.
The article is organized into five sections. Section~1 provides 
integration formulas for products of polynomials. Section~2
discusses monomial orders for multivariable polynomials and the 
well-known Vandermonde determinant, which serves as 
a guide for our proofs of Theorems~A and B in Section~3. We then 
describe a connection between our determinant result and the 
multinomial identity known as Dyson's conjecture in Section~4 
before finally proving Theorem~C in Section~5.

%%%%%%%%%%%%%%%%%%%%%%%%%%%%%%
\section{Integrating Products}
\label{sec:antiderivatives}
%%%%%%%%%%%%%%%%%%%%%%%%%%%%%%

We begin with a simple way to antidifferentiate products of polynomials. 

\begin{defn}[Derivative sequences]
	Let $R$ be a commutative ring. A sequence of polynomials $(f_\ast) = (f_0,f_1,f_2,\ldots)$ in $R[z]$ is a 	\emph{derivative sequence} if $f_n(z)$ has degree $n$ and $\frac{d}{dz} f_n(z) = f_{n-1}(z)$ for each positive integer $n$. For convenience, we define $f_n(z)$ to be the zero polynomial if $n$ is a negative integer.
\end{defn}

Derivative sequences are similar to \emph{Appell sequences},
which instead require the condition 
$\frac{d}{dz} f_n(z) = n\cdot f_{n-1}(z)$. The tools presented 
in this section appear similar to others in this related context; 
see \cite{lee11} and \cite{liupanzhang14} for examples.

\begin{example}[A special derivative sequence]\label{ex:derivative-sequence}
	If we fix some $z_0\in \C$ and define $f_n(z) = \frac{1}{n!}(z-z_0)^n$, 
	then $(f_\ast)$ is a derivative sequence in $\C[z]$ with the unusual 
	property that $f_n(z)$ is a factor of $f_m(z)$ for all $m \geq n$.
\end{example}

Every polynomial belongs to a derivative sequence, and the product of 
two polynomials has an antiderivative which is easily expressed using 
the derivative sequences to which they belong.

\begin{example}[Derivative sequences and antiderivatives]\label{ex:antiderivative}
	Let $f_5$ and $g_3$ be polynomials of degree $5$ and $3$ 
	respectively and let $(f_\ast)$ and $(g_\ast)$ be derivative 
	sequences that contain them.  We can use these  derivative 
	sequences to produce two antiderivatives for the product 
	$f_5\cdot g_3$.  Consider 
	$(f_6\cdot g_3) - (f_7\cdot g_2) + (f_8\cdot g_1) - (f_9\cdot g_0)$. 
	To see this is an antiderivative, simply expand 
	$\frac{d}{dz} [(f_6\cdot g_3) - (f_7\cdot g_2) 
	+ (f_8\cdot g_1) - (f_9\cdot g_0)]$
	to 
	$(f_5\cdot g_3 + f_6\cdot g_2) - (f_6\cdot g_2 + f_7\cdot g_1) 
	+ (f_7\cdot g_1 + f_8\cdot g_0) - (f_8\cdot g_0 + f_9\cdot g_{-1})$ 
	which simplifies to 
	$f_5\cdot g_3 - f_9\cdot g_{-1} = f_5\cdot g_3$ since $g_{-1}$ 
	is the zero polynomial. Switching the role of $f$ and $g$, 
	we see that 
	$(f_5\cdot g_4) - (f_4\cdot g_5) + (f_3\cdot g_6) 
	- (f_2\cdot g_7) + (f_1\cdot g_8) - (f_0\cdot g_9)$ 
	is another antiderivative of $f_5\cdot g_3$.  
\end{example}

The antiderivatives of a product listed in Example~\ref{ex:antiderivative} 
are just the result of iterated integration by parts.  The derivative 
sequences merely predetermine the antiderivatives that one uses.

\begin{lem}[An antiderivative formula]\label{lem:two-anti}
	If $(f_\ast)$ and $(g_\ast)$ are derivative sequences, then
	$r(z) = \sum_{i=0}^b (-1)^i (f_{a+1+i}(z)\cdot g_{b-i}(z))$
	is an antiderivatve of $f_a(z)\cdot g_b(z)$.  Note that if 
	$(f_\ast)$ is the derivative sequence with 
	$f_n(z) = \frac{1}{n!}(z-z_0)^n$, then $(z-z_0)^{a+1}$ is a 
	factor of the polynomial $r(z)$.
\end{lem}

\begin{proof}
    Computing the derivative is straightforward. We apply the product rule, split into
    two summands, and reindex; the result is that all terms cancel except the first,
    which is what we want, and the last, which is zero. In symbols:
    \[
        \frac{d}{dz} \left(	\sum_{i=0}^b (-1)^i (f_{a+1+i}\cdot g_{b-i}) \right)
		= \sum_{i=0}^b (-1)^i (f_{a+i}\cdot g_{b-i} + f_{a+1+i}\cdot g_{b-i-1})
    \]
	\begin{align*}
		&= \left( \sum_{i=0}^b (-1)^i (f_{a+i}\cdot g_{b-i}) \right) +
		\left( \sum_{i=0}^{b} (-1)^i (f_{a+i+1}\cdot g_{b-i-1}) \right) \\
		&= \left( \sum_{i=0}^b (-1)^i (f_{a+i}\cdot g_{b-i}) \right) -
		\left( \sum_{i=1}^{b+1} (-1)^i (f_{a+i}\cdot g_{b-i}) \right) \\
		&= f_a\cdot g_b - f_{a+b+1}\cdot g_{-1} \\ 
		&= f_a\cdot g_b
	\end{align*}
	The final assertion follows from the fact that 
	$(z-z_0)^{a+1}$ is a factor of $f_j(z)$ for all $j> a$ in this special case.
\end{proof}

For later use we include a specific application of Lemma~\ref{lem:two-anti}.

\begin{prop}\label{prop:two-factor-comp}
    Let $z_0$ be a complex variable and let $a$ and $b$ be positive 
    integers with $a+b=n$. If $p(z) = \int_0^z w^a (w-z_0)^b dw$, 
    then $p(z_0) = (-1)^b\frac{a!\, b!}{(n+1)!} z_0^{n+1}$ and 
    $\frac{\partial}{\partial z_0}(p(z_0)) 
    = (-1)^b\frac{a!\, b!}{n!} z_0^n$.
\end{prop}

\begin{proof}
    For any nonnegative integers $i$ and $j$, define the polynomials
    $f_i(z) = \frac{1}{i!}z^i$ and $g_j(z) = \frac{1}{j!}(z-z_0)^j$. Then $(f_\ast)$ 
    and $(g_\ast)$ are derivative sequences and by Lemma~\ref{lem:two-anti},
    we know that $f_{a+1}\cdot g_b - f_{a+2}\cdot g_{b-1} + \cdots + 
    (-1)^b f_{n+1}\cdot g_0$ is an antiderivative for the product 
    $f_a\cdot g_b$. Notice that $p(z) = a!\, b!\int_0^z (f_a(w)\cdot g_b(w))\ dw$
    and $f_i(0) = 0$ for all $i > 0$, so we have that
    \[
        p(z) = a!\, b!\big(f_{a+1}(z)\cdot g_b(z) - f_{a+2}(z)\cdot g_{b-1}(z) + 
        \cdots + (-1)^b f_{n+1}(z)\cdot g_0(z)\big).
    \]
    Finally, we note that $g_j(z_0) = 0$ for all $j>0$ while $g_0(z_0) = 1$, so 
    \[
        p(z_0) = a!\, b! (-1)^b (f_{n+1}(z_0)\cdot g_0(z_0)) 
        = (-1)^b \frac{a!\, b!}{(n+1)!} z_0^{n+1}
    \]
    as desired.
\end{proof}

The following example outlines our primary motivation for this section.

\begin{example}[Emphasizing a factor]\label{ex:gen-anti}
    Let $R$ be the polynomial ring $\C[z_1,\ldots,z_m]$ and
    let $\ba = (a_1,\ldots,a_m)$ be an $m$-tuple of positive 
    integers with $n = a_1+\cdots+a_m$. Define the
    polynomial $\tilde{p}_\ba(z) \in R[z]$ by
	\[
	    \tilde{p}_\ba(z) = \int_0^z 
		\frac{(w-z_1)^{a_1}}{a_1!} \cdots \frac{(w-z_m)^{a_m}}{a_m!}\ dw.
	\] 
	Notice that $(a_1!\cdots a_m!)\tilde{p}_\ba(z)$ is equal to the
	polynomial $p_\ba(z)$ defined in the introduction.
	For any choice of $j\in [m]$, define two derivative sequences
	$(f_\ast)$ and $(g_\ast)$; define $(f_\ast)$ explicitly by 
	$f_k(w) = \frac{1}{k!}(w-z_j)^k$ for each nonnegative integer $k$, and 
	let $(g_\ast)$ be any derivative sequence containing the term
	\[g_{n-a_j}(w) = \frac{(w-z_1)^{a_1}}{a_1!}
	\cdots\frac{(w-z_{j-1})^{a_{j-1}}}{a_{j-1}!}
	\frac{(w-z_{j+1})^{a_{j+1}}}{a_{j+1}!}
	\cdots \frac{(w-z_m)^{a_m}}{a_m!}.\]
	We write $\tilde{p}_{\ba,j}$ to denote the antiderivative of the 
	product $f_{a_j}\cdot g_{n-a_j}$ given by 
	Lemma~\ref{lem:two-anti}. Then 
	$\tilde{p}_{\ba}(z)= \tilde{p}_{\ba,j}(z) - \tilde{p}_{\ba,j}(0)$
	and we say that the index $j$ has been ``emphasized''. 
	We make several observations about these expressions:
    \begin{enumerate}
        \item $\tilde{p}_{\ba,j}(z_j) = 0$
        \item $\tilde{p}_{\ba,j}(0)$ is divisible by $z_j^{a_j+1}$
        \item $\tilde{p}_{\ba,j}(z_i)$ is divisible by $(z_i - z_j)^{a_i + a_j + 1}$
    \end{enumerate}
    To see this, first recall by Lemma~\ref{lem:two-anti} that $\tilde{p}_{\ba,j}(w)$ 
    is divisible by $(w - z_j)^{a_j + 1}$. From here, we can substitute $w = z_j$
    or $w = 0$ to prove (1) or (2), respectively. As for (3), notice that since
    $(w-z_i)^{a_i}$ is a factor of $g_{n-a_j}$, we know that $g_{n-a_j-\ell}(z_i) = 0$
    for all $\ell < a_i$. Therefore, the nonzero terms of $\tilde{p}_{\ba,j}(z_i)$
    are those of the form 
    $(-1)^\ell(f_{a_j + \ell + 1}(z_i) \cdot g_{n-a_j - \ell}(z_i))$
    where $\ell \geq a_i$, and in this setting, each $f_{a_j+\ell+1}(z)$ 
    is divisible by $(z_i-z_j)^{a_i+a_j+1}$, so we are done.
\end{example}

\begin{rem}[An explicit antiderivative]
    One can iteratively apply Lemma~\ref{lem:two-anti} to obtain an explicit
    expression for $\tilde{p}_{\ba,j}$. Concretely, 
    let $\ba = (a_1,\ldots,a_m)$ be an $m$-tuple of positive integers
    and fix $j\in [m]$. It can then be shown that 
    \[
        \tilde{p}_{\ba,j}(z) =
		\sum_{i \geq 0}
		(-1)^i 	\frac{(z-z_j)^{a_j+i+1}}{(a_j+i+1)!} 
		\left(\sum_{\substack{b_1+\cdots+b_m=i\\b_j=0}} 
	    \binom{i}{b_1,\ldots,b_m}
		\prod_{\substack{\ell \in [m]\\ \ell\neq j}} \frac{(z-z_\ell)^{a_\ell-b_\ell}}{(a_\ell-b_\ell)!}
		\right)
	\]
	although we will not require such explicit detail in this article.
\end{rem}

%%%%%%%%%%%%%%%%%%%%%%%%%%%%%%%%%%%%%%%%%%%%%%%%%%%%%%%%%%%%
\section{Ordering Monomials and the Vandermonde Determinant}
\label{sec:monomials}
%%%%%%%%%%%%%%%%%%%%%%%%%%%%%%%%%%%%%%%%%%%%%%%%%%%%%%%%%%%%

In this section we review how one computes the Vandermonde determinant 
as a warmup to the proofs of Theorems~\ref{thma:beardon-det} 
and \ref{thma:main-specific} in Section~\ref{sec:determinants}.  
Our proof uses the idea of a monomial order borrowed from the 
standard construction of a Gr\"obner basis. 

\begin{defn}[Monomial order]\label{def:monomial-order}
    Fix a positive integer $m$ and a polynomial
    ring $R = \mathbb{C}[z_1,\ldots,z_m]$. Ordering the variables
    $z_1 < z_2 < \cdots < z_m$ then yields a lexicographic order 
    on the monomials of $R$. One first compares the degree of $z_m$ 
    in the two terms and orders them accordingly.  If they agree, 
    one compares the degrees of $z_{m-1}$ and so on.  For any 
    polynomial $f \in R$, the \emph{leading term} $\lt(f)$ is the 
    summand of $f$ containing the largest monomial in this order 
    and the coefficient of the leading term is the \emph{leading coefficient}.
\end{defn}

\begin{example}
    Let $R = \mathbb{C}[z_1,z_2,z_3]$ and define $f\in R$ by 
    $f = 7z_1^9z_2^6z_3^3 -  6z_1z_2^2z_3^4$. Then 
    $z_1^9z_2^6z_3^3 < z_1z_2^2z_3^4$ in the lexicographic order
    on $R$, so $\lt(f) = -6z_1z_2^2z_3^4$ and the leading coefficient is $-6$.
\end{example}

The notion of a leading term is useful in the evaluation of the
Vandermonde determinant, traditionally introduced to show that
any $n+1$ distinct points on the rational normal curve
are in general position.
Recall that if $\mathbf{M}$ is an $n\times n$ matrix, then the 
\emph{Leibniz determinant formula} computes the determinant of $\mathbf{M}$ 
as a sum over permutations $\sigma$ in the symmetric group $\sym_n$:
\[
    \det \mathbf{M} = \sum_{\sigma\in \sym_n} \sgn(\sigma) 
    \prod_{i\in [n]} (\mathbf{M})_{\sigma(i),i}
\] 

\begin{thm}[Vandermonde determinant]\label{thm:vandermonde}
If $\mathbf{V}$ is the Vandermonde matrix
\[ \mathbf{V} = \begin{pmatrix}
	1 & 1 & 1 & \cdots & 1 & 1\\
	z_1 & z_2 & z_3 & \cdots & z_{n-1} & z_n \\
	z_1^2 & z_2^2 & z_3^2 & \cdots & z_{n-1}^2 & z_n^2 \\
	\vdots & \vdots & \vdots & \ddots & \vdots & \vdots \\
	z_1^{n-2} & z_2^{n-2} & z_3^{n-2} & \cdots & z_{n-1}^{n-2} & z_n^{n-2} \\
	z_1^{n-1} & z_2^{n-1} & z_3^{n-1} & \cdots & z_{n-1}^{n-1} & z_n^{n-1} \\
	\end{pmatrix}\]
with $(i,j)$-entry $z_j^{i-1}$, then its determinant is 
	\[
	\det \mathbf{V}= \prod_{\substack{j,k \in[n]\\ j<k}} (z_k - z_j).
	\]
\end{thm}

\begin{proof}
    We view the entries as monomials in the variables $z_1,\ldots,z_n$ 
    and we let $D(\bz) \in \C[z_1,\ldots,z_n]$ denote the multivariable 
    polynomial on the righthand side of the claimed formula. 
    Since the entries in the $i$-th row of $\mathbf{V}$ are homogeneous
    polynomials of degree $i-1$, we know that each
    summand in the Leibniz formula for $\det \mathbf{V}$
    is homogeneous of degree $0+1+\cdots +(n-1)$ and thus $\det \mathbf{V}$
    is a homogeneous polynomial of degree at most $\binom{n}{2}$.
    Moreover, the determinant is unchanged if the $j$-th column is 
    subtracted from the $k$-th column, and since every entry in the 
    new $k$-th column is divisible by $z_k-z_j$, this expression 
    divides the determinant. Since the linear factors that arise as 
    $j$ and $k$ are varied represent distinct non-associate prime 
    elements of the polynomial ring $\mathbb{C}[z_1,\ldots,z_n]$, 
    we know that their product, $D(\bz)$, divides $\det \mathbf{V}$ 
    as well. Because $D(\bz)$ is also homogeneous with degree 
    $\binom{n}{2}$, we know that for some constant $C$, 
    $\det \mathbf{V}$ is of the form 
    $\det \mathbf{V} = C \cdot D(\bz)$. 
    
    To determine the value of $C$, we lexicographically order the 
    monomials as in Definition~\ref{def:monomial-order}
    and then compare the leading terms of $\det \mathbf{V}$ and $D(\bz)$.
    In the latter case, the leading term is $\prod_{k\in [n]} z_k^{k-1}$,
    obtained by always choosing $z_k$ instead of $z_j$ when expanding the
    product. Similarly, the leading term of $\det \mathbf{V}$ appears
    in only one summand of the Leibniz formula,
    corresponding to when $\sigma$ is the identity permutation. To 
    see this, first notice that $(\mathbf{V})_{nn}$ is the only 
    entry which contributes a large enough power of $z_n$ to the 
    determinant, so we must have $\sigma(n) = n$. We then induct by 
    observing that the $(n-1)\times (n-1)$ submatrix of $\mathbf{V}$ 
    obtained by deleting the last row and column is simply a smaller 
    Vandermonde	matrix. Therefore, $\lt(\det \mathbf{V})$
    is also $\prod_{k\in [n]} z_k^{k-1}$, the product of the diagonal
    entries. So $C=1$ and we are done.    
\end{proof}

In more complicated cases, we can use the Leibniz formula
to compute the leading term of a determinant.

\begin{prop}[Leading terms of determinants]\label{prop:lt-det}
    Let $R = \mathbb{C}[z_1,\ldots,z_m]$ be lexicographically ordered and
    let $\mathbf{M}$ be a $k\times k$ matrix with entries in $R$. If
    $\mathbf{M}^\prime$ is the $k\times k$ matrix with $(i,j)$-entry
    $(\mathbf{M}^\prime)_{ij} = \lt((\mathbf{M})_{ij})$, then 
    $\lt(\det \mathbf{M}) = \lt(\det \mathbf{M}^\prime)$.
\end{prop}

\begin{proof}
    The proof is immediate from the Leibniz formula and the fact that
    the leading term of a product is the product of the leading terms.
\end{proof}

%%%%%%%%%%%%%%%%%%%%%%%%%%%%%%%%
\section{A Determinant Identity}
\label{sec:determinants}
%%%%%%%%%%%%%%%%%%%%%%%%%%%%%%%%

While studying a question about the critical values of complex
polynomials, Beardon, Carne, and Ng prove the following theorem.

\begin{thm}[\cite{beardon02}]\label{thm:beardon-det}
	For any $\bz = (z_1,\ldots,z_n) \in \mathbb{C}^n$, define 
	\[p(z) = \displaystyle\int_0^z (w-z_1)\cdots (w-z_n)\ dw\]
	and let $\theta\colon\mathbb{C}^n \to \mathbb{C}^n$ be the 
	map which sends $(z_1,\ldots,z_n)$ to $(p(z_1),\ldots,p(z_n))$.
	Then the Jacobian matrix for $\theta$ is
	the $n\times n$ matrix $\bj$ with $(i,j)$-entry 
	$(\bj)_{ij} = \frac{\partial}{\partial z_i} p(z_j)$ and
	$\bj$ is invertible if $z_1,\ldots,z_n$
	are distinct and nonzero.
\end{thm}

We generalize this result to the setting where $z_1,\ldots,z_n$ are nondistinct.

\begin{defn}\label{def:gen-poly}
    Let $\bz = (z_1,\ldots,z_m) \in \mathbb{C}^m$ and
    let $\ba = (a_1,\ldots,a_m)$ be an $m$-tuple of positive 
    integers. Define the polynomial $p_\ba\in \C[z]$ by
    \[
        p_\ba(z) = p_{\bz,\ba}(z) = 
        \int_0^z (w-z_1)^{a_1}\cdots (w-z_m)^{a_m}\ dw.
    \]
    Notice that each $z_i$ is a critical point for $p_\ba$ with 
    multiplicity $a_i$, and if we let $n=a_1 + \cdots + a_m$, then 
    $p_\ba$ is a polynomial of degree $n+1$.  The reader should 
    note that one can view $z_1,\ldots,z_m$ as complex variables
    (instead of specific complex numbers) and then regard $p_\ba$ 
    as a polynomial in the ring $R[z]$, where 
    $R = \C[z_1,\ldots,z_m]$. In the remainder of the article, we 
    switch between these two viewpoints as needed.  With this in 
    mind, define $\theta_\ba\colon \mathbb{C}^m \to \mathbb{C}^m$ 
    to be the map which sends $(z_1,\ldots,z_m)$ to 
    $(p_\ba(z_1),\ldots,p_\ba(z_m))$. Then the Jacobian matrix 
    for $\theta_\ba$ is the $m\times m$ matrix $\bj_\ba$ with 
    $(i,j)$-entry 
    $(\bj_\ba)_{ij} = \frac{\partial}{\partial z_i} p_\ba(z_j)$.
\end{defn}

The next result factors the determinant of $\bj_{\ba}$ and 
proves Theorem~\ref{thma:main-specific}.

\begin{thm}\label{thm:general-det}
    Let $\ba = (a_1,\ldots,a_m)$ be an $m$-tuple of positive integers.
	Then the determinant of $\bj_{\ba}$ is
	\[
	    \det \bj_{\ba} = \frac{1}{\binom{n}{a_1,\ldots,a_m}}
	    \left( \prod_{j \in [m]} (-z_j)^{a_j} \right)
	    \left( \prod_{\substack{j,k\in [m] \\ j\neq k}} 
	    (z_k-z_j)^{a_j} \right).
	\]
	Hence, $\bj_{\ba}$ is invertible if and only if $z_1,\ldots,z_k$
	are distinct and nonzero.
\end{thm}

Setting each $a_i = 1$ provides the following corollary, which proves 
Theorem~\ref{thma:beardon-det} and immediately yields an
alternate proof for Theorem~\ref{thm:beardon-det}.

\begin{cor}\label{cor:simple-det}
	If $\ba$ is the $n$-tuple $(1,\ldots,1)$, then
	$\bj_{\ba}$ is the Jacobian matrix $\bj$ defined in the introduction and
	\[\det \bj_{\ba} = \frac{1}{n!} \left( \prod_{j\in [n]} (-z_j)\right) 
	\left( \prod_{\substack{j,k\in [n]\\ j\neq k}} (z_k-z_j)\right).\]
\end{cor}

\begin{example}[Two variables with multiplicity]
	It is easy to compute the determinant in small cases using 
	a software package such as \emph{SageMath}.  Consider, for example, 
	the case where $n=2$ and $\ba = (2,3)$, and for readability we 
	write $\bz = (x,y)$ instead of $\bz=(z_1,z_2)$.  First we compute 
	the polynomial $p_\ba(z)$, with the details omitted:
	\begin{align*}
        p_\ba(z) =& \int_0^z (w-x)^2(w-y)^3\ dw \\
        =&\ \frac{z^6}{6} - (2x +3y)\frac{z^5}{5} + (x^2 + 6xy + 3y^2) \frac{z^4}{4} \\
        & -(3x^2 y + 6xy^2 + y^3) \frac{z^3}{3} + (3x^2y^2 + 2xy^3) \frac{z^2}{2}
        -(x^2 y^3) z
	\end{align*}
	Next we compute the map $\theta_\ba\colon \C^2\to \C^2$:
	\begin{align*}
	    \theta_\ba(x,y) &= (p_\ba(x),p_\ba(y)) \\
	    &= \left(\frac{1}{60} x^{6} - \frac{1}{10} x^{5} y + 
	    \frac{1}{4} x^{4} y^{2} - \frac{1}{3} x^{3} y^{3},
	    -\frac{1}{4} x^{2} y^{4} + \frac{1}{10} x y^{5} - \frac{1}{60} y^{6}\right)
	\end{align*}
	Then we compute the Jacobian $\bj_{\ba}$:
	\[\bj_{\ba} = 
    	\begin{pmatrix}
    	\frac{1}{10} \, x^{5} - \frac{1}{2} \, x^{4} y + x^{3} y^{2} - x^{2} y^{3} & 
    	-\frac{1}{2} \, x y^{4} + \frac{1}{10} \, y^{5} \\[4pt]
    	-\frac{1}{10} \, x^{5} + \frac{1}{2} \, x^{4} y - x^{3} y^{2} &
    	-x^{2} y^{3} + \frac{1}{2} \, x y^{4} - \frac{1}{10} \, y^{5}
    	\end{pmatrix}
	\]
	And finally we compute the determinant of $\bj_{\ba}$:
	\[\det \bj_{\ba} = -\frac{1}{10}x^2y^3(x-y)^5\]
    One could instead compute this determinant from three observations. 
    First, notice that the first column is divisible by $x^2$ and the 
    second column by $y^3$. Next, if we subtract the second column from 
    the first and factor, then each resulting entry is divisible by 
    $(x-y)^5$. Thus $\det \bj_{\ba}$ is a homogeneous polynomial of 
    degree $10$ which is divisible by $x^2 y^3 (x-y)^5$, so 
    \[ \det \bj_{\ba} = C x^2 y^3 (x-y)^5 \] 
    for some constant $C$. By inspection, the leading term of
    $x^2 y^3 (x-y)^5$ is $-x^2y^8$ while the leading term of
    $\det \bj_\ba$ is $\frac{1}{10}x^2y^8$. Therefore, $C = \frac{-1}{10}$
    and we are done.
\end{example}

To emphasize its straightforward structure, we give the proof of
Theorem~\ref{thm:general-det} here before continuing on to 
prove the prerequisite lemmas. 

\begin{proof}[Proof of Theorem~\ref{thm:general-det}]
    The proof is in three steps. First, Proposition~\ref{prop:col-divisible}
    tells us that $\det \bj_{\ba}$ is divisible by $z_j^{a_j}$ 
    for each $j\in [m]$. Second, Proposition~\ref{prop:col-diff-divisible}
	implies that the determinant is divisible by
	$(z_k-z_j)^{a_j+a_k}$ for each $j<k$ in $[m]$. 
	Since these polynomials are powers of distinct 
	non-associate prime elements in
	the polynomial ring $\mathbb{C}[z_1,\ldots,z_m]$ and we 
	know that $\det \bj_{\ba}$ is a homogeneous polynomial of 
	degree at most $n(a_1+\cdots+a_m)$, it must be that $\det \bj_{\ba}$ 
	is a constant multiple of their product. After rewriting 
	$(z_k - z_j)^{a_j + a_k} = (-1)^{a_k}(z_k - z_j)^{a_j}(z_j-z_k)^{a_k}$ 
	and absorbing powers of $-1$ into the constant, we may write
	$\det \bj_{\ba} = C\cdot D(\bz)$, where $C$ is a constant and
	\[
	    D(\bz) = \left(\prod_{j\in [m]} (-z_j)^{a_j} \right)
	    \left(\prod_{\substack{j,k\in [m] \\ j\neq k}} 
	    (z_k-z_j)^{a_j}\right).
	\]
	Lastly, we determine the value of $C$ by comparing the
	leading coefficients of $\det \bj_\ba$ and $D(\bz)$. We can see 
	that $\lt(D(\bz))$ has coefficient $(-1)^{a_1+2a_2+\cdots+ma_m}$, 
	obtained by always choosing the higher-index terms when expanding 
	the product $(z_k-z_j)^{a_j}$. On the other hand, we know by 
	Proposition~\ref{prop:constant-term} that $\lt(\det \bj_\ba)$ 
	has coefficient 
	$(-1)^{a_1+2a_2+\cdots+ma_m}\binom{n}{a_1,\ldots,a_m}^{-1}$. 
	Thus, $C = \binom{n}{a_1,\ldots,a_m}^{-1}$ and we are done.	
\end{proof}

We now prove several lemmas and the three propositions used in
the proof of Theorem~\ref{thm:general-det}. The first two
propositions are straightforward, while the third is more
complicated.

\begin{lem}[Jacobian entries]\label{lem:jacobian-entries}    
    Let $\bz = (z_1,\ldots,z_m)\in \mathbb{C}^m$ and let
	$\ba = (a_1,\ldots,a_m)$ be an $m$-tuple of positive integers.
	Then for any $i,j\in [m]$ and for any choice of $k\in [m]$, 
	the $(i,j)$-entry of the Jacobian matrix $J_{\ba}$ can be written as
	\[
	    (\bj_\ba)_{ij} = \frac{\partial}{\partial z_i} p_\ba(z_j) 
	    = (a_1! \cdots a_m!) \frac{\partial}{\partial z_i} \Big[
	    \tilde{p}_{\ba,k}(z_j) - \tilde{p}_{\ba,k}(0)\Big].
	\]
\end{lem}

\begin{proof}
    This follows immediately from Example~\ref{ex:gen-anti}
    and Definition~\ref{def:gen-poly}.
\end{proof}

\begin{prop}[Columns]\label{prop:col-divisible}
	Each entry of the $j$-th column of $\bj_{\ba}$ is divisible by~$z_j^{a_j}$.
\end{prop}

\begin{proof}
    Let $i,j\in [m]$. By Lemma~\ref{lem:jacobian-entries},
    emphasizing $z_j$, we have the following:
    \[
        (\bj_\ba)_{ij} = (a_1! \cdots a_m!) \frac{\partial}{\partial z_i} 
        \Big[\tilde{p}_{\ba,j}(z_j) - \tilde{p}_{\ba,j}(0)\Big]
    \]
    By the observations made in Example~\ref{ex:gen-anti},
    we know that $\tilde{p}_{\ba,j}(z_j) = 0$ and $\tilde{p}_{\ba,j}(0)$ is divisible
    by $z_j^{a_j+1}$. Thus, $(\bj_\ba)_{ij}$ is divisible by $z_j^{a_j}$ 
    if $i = j$ and $z_j^{a_j+1}$ otherwise.
\end{proof}

\begin{prop}[Column differences]\label{prop:col-diff-divisible}
    If the $j$-th column of $\bj_{\ba}$ is subtracted from the
    $k$-th column, then each entry of the new $k$-th column is
	divisible by $(z_k-z_j)^{a_k+a_j}$.
\end{prop}

\begin{proof}
    For $i,j,k\in [m]$ we show that $(\bj_{\ba})_{ik} - (\bj_{\ba})_{ij}$ is divisible by
    $(z_k-z_j)^{a_k+a_j}$. Using Lemma~\ref{lem:jacobian-entries},
    emphasizing $z_j$ in both applications, we have the following:
    \[
        \frac{(\bj_{\ba})_{ik} - (\bj_{\ba})_{ij}}
        {a_1! \cdots a_m!} =
        \frac{\partial}{\partial z_i}
        \Big[ \tilde{p}_{\ba,j}(z_k) - \tilde{p}_{\ba,j}(0)
        \Big] -
        \frac{\partial}{\partial z_i}
        \Big[ \tilde{p}_{\ba,j}(z_j) - \tilde{p}_{\ba,j}(0)
        \Big]
    \]
    By canceling terms, we have that
    \[
        (\bj_{\ba})_{ik} - (\bj_{\ba})_{ij} = 
        (a_1! \cdots a_m!)\frac{\partial}{\partial z_i}
        \Big[\tilde{p}_{\ba,j}(z_k) - \tilde{p}_{\ba,j}(z_j)\Big],
    \]
    and by Example~\ref{ex:gen-anti}, we know that 
    $\tilde{p}_{\ba,j}(z_k)$ is divisible by 
    $(z_k - z_j)^{a_k+a_j+1}$ and
    $\tilde{p}_{\ba,j}(z_j) = 0$. Therefore, 
    $(\bj_{\ba})_{ik} - (\bj_{\ba})_{ij}$ is divisible
    by $(z_k - z_j)^{a_k+a_j}$ if $i\in \{j,k\}$ and
    $(z_k - z_j)^{a_k+a_j+1}$ otherwise, so we are done.
\end{proof}

The last proposition used in the proof of Theorem~\ref{thm:general-det}
concerns the coefficient of the leading term in $\det \bj_\ba$. As a first
step, we consider the highest power of $z_m$ which appears in each 
entry of $\bj_\ba$ and apply this to the leading term of $\det \bj_\ba$.

\begin{lem}[Exponents of $z_m$]\label{lem:zm-exponents}
    Let $\ba = (a_1,\ldots,a_m)$ be an $m$-tuple of positive integers.
    Then the highest exponent of $z_m$ appearing in the entry $(\bj_\ba)_{ij}$
    is $a_m$ if $i,j < m$, $a_m-1$ if $j < i = m$, and $n$ if $j = m$.
\end{lem}

\begin{proof}
    By Definition~\ref{def:gen-poly}, we have the following:
    \begin{align*}
        (\bj_\ba)_{ij} &= \frac{\partial}{\partial z_i} p_\ba(z_j) \\
        &= \frac{\partial}{\partial z_i} \int_0^{z_j}
        (w - z_1)^{a_1}\cdots (w-z_m)^{a_m}\ dw
    \end{align*}
    We now expand inside the integral to find the largest power of $z_m$.
    When $i=m$, the largest power appears in the term
    $\frac{\partial}{\partial z_m}\int_0^{z_j}
    w^{n-a_m}(-z_m)^{a_m}\ dw$, and thus the maximum exponent is $a_m-1$ if
    $j < m$ and $n$ if $j = m$. Similarly, when $i<m$, the largest power
    of $z_m$ appears in the term
    $\frac{\partial}{\partial z_i}\int_0^{z_j}
    w^{n-a_m-1}(-z_i)(-z_m)^{a_m}\ dw$,
    and so the maximum exponent is $a_m$ if $j<m$ and $n$ if $j=m$.
\end{proof}

\begin{example}
    Let $\ba = (3,7,2,6)$. Then the highest exponent of $z_4$ in the entry
    $(\bj_\ba)_{ij}$ is $6$ if $i,j<4$, $5$ if $j< i= 4$, and
    $18$ if $j = 4$. This information is easily visualized in the
    matrix
    \[
        \begin{pmatrix}
            6 & 6 & 6 & 18 \\
            6 & 6 & 6 & 18 \\
            6 & 6 & 6 & 18 \\
            5 & 5 & 5 & 18
        \end{pmatrix}
    \]
    where each entry contains the largest exponent of $z_4$ in the corresponding
    entry of $\bj_\ba$. As a consequence, it is easy to see that the appearance of
    $\lt(\det \bj_\ba)$ in the Leibniz formula is restricted to the summands
    which include the entry $(\bj_\ba)_{44}$.
\end{example}

\begin{lem}[Leading term of $\det \bj_\ba$]\label{lem:lt-det-ja}
    Let $\ba = (a_1,\ldots,a_m)$ be an $m$-tuple of positive integers.
    Then the leading term of $\det \bj_\ba$  
    may be computed recursively via the formula
    \[
        \lt(\det \bj_\ba) =(-z_m)^{a_m(m-1)}\cdot  
         \lt(\det \bj_{\ba^\prime})\cdot
         \lt((\bj_\ba)_{mm})
    \]
    where $\ba^\prime = (a_1,\ldots,a_{m-1})$.
\end{lem}

\begin{proof}
    By Proposition~\ref{prop:lt-det}, the leading term of
    $\det \bj_\ba$ is equal to the leading term of the 
    determinant of the matrix obtained by taking the leading 
    term of each entry in $\bj_\ba$. Since $z_m$ is the 
    highest-ordered variable, we know that
    $\lt((\bj_\ba)_{ij})$ has the highest power of $z_m$ in $(\bj_\ba)_{ij}$
    as a factor.
    Thus, Lemma~\ref{lem:zm-exponents} and the Leibniz formula tell us
    that $\lt(\det \bj_\ba)$ has a factor of $z_m^{a_m(m-1) + n}$,
    and this term is obtained from the Leibniz formula only in terms
    which include $(\bj_\ba)_{mm}$. As for the entries with $i,j < m$, 
    \begin{align*}
        \lt((\bj_\ba)_{ij}) &= \lt\left(
        \frac{\partial}{\partial z_i} \int_0^{z_j}
        (w - z_1)^{a_1}\cdots (w-z_m)^{a_m}\ dw \right) \\
        &= \lt\left(\frac{\partial}{\partial z_i}\int_0^{z_j} 
        (w - z_1)^{a_1}\cdots (w-z_{m-1})^{a_{m-1}}(-z_m)^{a_m}\ dw\right) \\
        &= \lt\left((-z_m)^{a_m} \frac{\partial}{\partial z_i}\int_0^{z_j} 
        (w - z_1)^{a_1}\cdots (w-z_{m-1})^{a_{m-1}}\ dw\right) \\
        &= (-z_m)^{a_m}\lt((\bj_{\ba^\prime})_{ij})
    \end{align*}
    and so by the Leibniz formula, the claim is proven.
\end{proof}

We are now ready to compute the coefficient of the leading term of $\det \bj_\ba$.

\begin{prop}[Constant coefficient]\label{prop:constant-term}
    Let $\ba = (a_1,\ldots,a_m)$ be an $m$-tuple of positive integers. 
    Then $\lt(\det \bj_\ba)$ has coefficient
    $(-1)^{a_1+2a_2+\cdots +ma_m} \binom{n}{a_1,\ldots,a_m}^{-1}$.
\end{prop}

\begin{proof}
	We prove by induction on $m$. When $m=1$, we have $\ba = (a_1)$
	and 
	\[p_\ba(z) = \int_0^z (w-z_1)^{a_1}\ dw = \frac{1}{a_1+1}
	\Big[ (z-z_1)^{a_1+1} - (-z_1)^{a_1+1} \Big],\]
	so $\bj_\ba$ is a $1\times 1$ matrix 
	with a single entry:
	\[
	    (\bj_\ba)_{1,1} = \frac{d}{d z_1} \Big[p_\ba(z_1)\Big]
	    = \frac{d}{d z_1} \left[ \frac{(-1)^{a_1}z_1^{a_1+1}}{a_1+1} \right]
	    = (-1)^{a_1}z_1^{a_1}
	\]
	Thus $\lt(\det \bj_\ba) = (-1)^{a_1}z_1^{a_1}$ and so the coefficient is
	$(-1)^{a_1}$.
	
	Now, let $\ba = (a_1,\ldots,a_m)$ and suppose the claim holds 
	for the $(m-1)\times(m-1)$ matrix $\bj_{\ba^\prime}$, where 
	$\ba^\prime = (a_1,\ldots,a_{m-1})$. Then Lemma~\ref{lem:lt-det-ja}
	tells us that
	\[
        \lt(\det \bj_\ba) =(-z_m)^{a_m(m-1)}\cdot  
         \lt(\det \bj_{\ba^\prime})\cdot
         \lt((\bj_\ba)_{mm})
    \]
	and by Proposition~\ref{prop:two-factor-comp}, we have the following:
	\begin{align*}
	    \lt((\bj_\ba)_{mm}) 
	    &= \lt\left(\frac{\partial}{\partial z_m}\int_0^{z_m} 
        (w-z_1)^{a_1}\cdots(w-z_m)^{a_m}\ dw\right) \\
        &= \lt\left(\frac{\partial}{\partial z_m}\int_0^{z_m} 
        w^{n-a_m}(w-z_m)^{a_m}\ dw\right) \\
        &= (-1)^{a_m}\frac{a_m!\, (n-a_m)!}{n!}z_m^n
	\end{align*}
	So the leading coefficient of $(\bj_\ba)_{mm}$ is 
	\[
	    (-1)^{a_m}\frac{a_m!\, (n-a_m)!}{n!}
	\]
	and by the inductive hypothesis, the leading coefficient of 
	$\det \bj_{\ba^\prime}$ is
	\[
	    \left((-1)^{a_1+2a_2+\cdots +(m-1)a_{m-1}} 
	    \binom{a_1+\cdots+a_{m-1}}{a_1,\ldots,a_{m-1}}^{-1}\right).
	\]
	Putting it all together, we have that the leading coefficient of 
	$\det \bj_\ba$ is
	\[
	    (-1)^{a_1+2a_2+\cdots +ma_m}
	    \frac{a_1!\cdots a_{m-1}!}{(n-a_m)!}
	    \frac{a_m!(n-a_m)!}{n!}
	\]
	\[
	    = (-1)^{a_1+2a_2+\cdots +ma_m} \binom{n}{a_1,\ldots,a_{m}}^{-1}
	\]
	as desired.
\end{proof}

%%%%%%%%%%%%%%%%%%%%%%%%%%%%%%%%%%%%%%%%%%%
\section{Connections to Dyson's Conjecture}
\label{sec:dyson}
%%%%%%%%%%%%%%%%%%%%%%%%%%%%%%%%%%%%%%%%%%%

The polynomial expression presented in Theorem~\ref{thm:general-det} 
is connected to a well-known multinomial identity,
originally conjectured by the physicist Freeman Dyson \cite{dyson62}
and proven independently by Gunson \cite{gunson62} and Wilson
\cite{wilson62}. A particularly short proof was later given
by Good \cite{good70}.

\begin{thm}[Dyson's Conjecture]
    Let $a_1,\ldots,a_m$ be positive integers and let $n = a_1 + \cdots + a_m$.
	Then the constant term of the Laurent polynomial
	\[
	    \prod_{\substack{j,k\in [m] \\ j\neq k}} 
	    \left(1 - \frac{z_j}{z_k} \right)^{a_j} 
	\]
	is equal to the multinomial coefficient
	$\binom{n}{a_1, \ldots, a_m}$.
\end{thm}

This is connected to Theorem 3.2. Define $\ba = (a_1,\ldots,a_m)$ and consider:
\begin{align*}
    \det \bj_\ba &= \frac{1}{\binom{n}{a_1,\ldots,a_m}}
	\left( \prod_{j \in [m]} (-z_j)^{a_j} \right)
	\left( \prod_{\substack{j,k\in [m] \\ j\neq k}} (z_k-z_j)^{a_j} \right) \\
	&= \frac{1}{\binom{n}{a_1,\ldots,a_m}}
	\left( \prod_{j \in [m]} (-z_j)^{a_j} \right)
	\left( \prod_{\substack{j,k\in [m] \\ j\neq k}} 
	\left(z_k\left(1-\frac{z_j}{z_k}\right)\right)^{a_j} \right) \\
	&= \frac{(-1)^n}{\binom{n}{a_1,\ldots,a_m}}
	\left( \prod_{j\in [m]} (z_j)^{n} \right)
	\left( \prod_{\substack{j,k\in [m] \\ i\neq j}} 
	\left(1-\frac{z_j}{z_k}\right)^{a_j} \right)
\end{align*}
From here, we can see that Dyson's conjecture is equivalent
to the fact that the monomial $\prod_{j\in [m]} z_j^{n}$ appears 
in $\det \bj_\ba$ with coefficient $(-1)^n$. In other words, 
if we divide the $j$-th column of $\bj_\ba$ by the monomial 
$z_j^n$ for each $j$, then the determinant becomes a Laurent 
polynomial with constant term $(-1)^n$.

%%%%%%%%%%%%%%%%%%%%%%%%%%%%%%%%%%%%%%%%%%%%%
\section{A Stratification for $\mathbb{C}^n$}
\label{sec:stratification}
%%%%%%%%%%%%%%%%%%%%%%%%%%%%%%%%%%%%%%%%%%%%%

As discussed in Section~\ref{sec:determinants}, the map
$\theta\colon\mathbb{C}^n \to \mathbb{C}^n$ which sends
each $\bz = (z_1,\ldots,z_n)$ to $(p(z_1),\ldots,p(z_n))$
has $\bj$ as its Jacobian matrix. By Theorem~\ref{thm:beardon-det}
and Corollary~\ref{cor:simple-det}, $\bj$ is invertible if and 
only if the entries of $\bz$ are distinct and nonzero. Together 
with the Inverse Function Theorem (see \cite{gunningrossi65}, for 
example), this implies that $\theta$ is a local homeomorphism at $\bz$ 
if $z_1,\ldots,z_n$ are distinct and nonzero. In this section, we 
provide a similar interpretation for Theorem~\ref{thm:general-det}.

\begin{defn}[Partitions]
    Recall that if $\lambda$ is a collection of $\ell$ nonempty pairwise disjoint
    sets with union equal to a given set, then $\lambda$
    is a \emph{partition} of that set with $\ell$ \emph{blocks}.
    The partition $\lambda$ is a \emph{refinement} of the
    partition $\mu$ (or \emph{finer} than $\mu$) if each block in
    $\mu$ is a union of blocks in $\lambda$.
\end{defn}

\begin{defn}[Stratifying $\mathbb{C}^n$]
    For each $\bz = (z_1,\ldots,z_n) \in \mathbb{C}^n$, define 
    $\partition(\bz)$ to be the unique partition of $[n]$ such that 
    $i$ and $j$ belong to the same block of $\lambda$ if and only if $z_i = z_j$. Then, for
    each partition $\lambda$ of $[n]$, define the
    subspace 
    \[
        \mathbb{C}^\lambda = \{\bz \in \mathbb{C}^n 
        \mid \partition(\bz) = \lambda\}
    \]
    of $\mathbb{C}^n$. First, observe that $\C^\lambda$ and $\C^\mu$ 
    are disjoint subspaces if and only if $\lambda$ and $\mu$ are 
    distinct partitions, and that each point in $\mathbb{C}^n$ lies 
    in a unique $\C^\lambda$. In other words, these subspaces form a 
    partition of $\mathbb{C}^n$. Next, for each partition $\lambda$ 
    of $[n]$, the topological closure $\overline{\C^\lambda}$ is the 
    linear subspace of $\C^n$ consisting of all points where we only 
    require that $z_i = z_j$ if $i$ and $j$ belong to the same block. Thus, we see that $\overline{\C^\lambda}$ is the union of the disjoint 
    subspaces $\C^\mu$, where $\mu$ is a partition
    of $[n]$ and $\lambda$ is a refinement of $\mu$. Finally,    
    if $\bz \in \C^\lambda$, then $\theta(\bz) \in \overline{\C^\lambda}$.
\end{defn}

Both $\C^\lambda$ and $\overline{\C^\lambda}$ are familiar spaces. 
Fix a partition $\lambda = \{S_1,\ldots,S_\ell\}$ of $[n]$ and define a map 
$\phi: \overline{\C^\lambda} \to \mathbb{C}^\ell$ by sending
each $\bz \in \C^\lambda$ to $\phi(\bz) = (y_1,\ldots,y_\ell)$, where
each $y_i$ is the common value shared by entries in $\bz$ with indices from $S_i$.
For example, if $\lambda = \{\{1,3\},\{2,5\},\{4\}\}$ and 
$\bz = (a,b,a,b,b) \in \C^\lambda$, then we have $\phi(\bz) = (a,b,b)$. 
It is then clear that $\phi$ is a homeomorphism and the restriction of
$\phi$ to $\C^\lambda$ is a homeomorphism onto its image: the space
of all points in $\mathbb{C}^\ell$ with distinct coordinates. This image
is well-known as the complement of the \emph{complex braid arrangement},
i.e. the complement of the hyperplanes in $\mathbb{C}^\ell$ defined by the
equations $y_i = y_j$. If we denote the union of these hyperplanes by
$\mathcal{A}_\ell$, then we have that $\C^\lambda$ and $\overline{\C^\lambda}$
are homeomorphic via $\phi$ to 
$\mathbb{C}^\ell-\mathcal{A}_\ell$ and $\mathbb{C}^\ell$, respectively.

\begin{thm}[Local homeomorphisms]\label{thm:stratification}
    Let $\lambda$ be a partition of $[n]$ and define the map 
    $\theta_\lambda\colon  \C^\lambda\to \overline{\C^\lambda}$ 
    to be the corresponding restriction of $\theta$. Then 
    $\theta_\lambda$ is a local homeomorphism at $\bz \in \C^\lambda$ 
    so long as the entries of $\bz$ are nonzero.
\end{thm}

\begin{proof}
    Let $\lambda = \{S_1,\ldots,S_\ell\}$ be a partition of 
    $[n]$ with $a_i = |S_i|$ for each $i$ and define the map 
    $\phi$ as above. Since $\phi$ is a homeomorphism, we have 
    that $\theta_\lambda$ is a local homeomorphism at 
    $\bz \in \C^\lambda$ if and only if the map 
    $\phi\theta_\lambda\phi^{-1}\colon\mathbb{C}^\ell 
    - \mathcal{A}_\ell \to \mathbb{C}^\ell$ 
    is a local homeomorphism at $\phi(\bz)$. If we define 
    $\ba = (a_1,\ldots,a_\ell)$, we then have that 
    $\phi\theta_\lambda\phi^{-1}$ sends each 
    $\by = (y_1,\ldots,y_\ell) \in \mathbb{C}^\ell$ to 
    $(p_{\ba}(y_1),\ldots,p_{\ba}(y_\ell))$.
    In other words, $\phi\theta_\lambda\phi^{-1} = \theta_\ba$
    and by Theorem~\ref{thm:general-det}, the associated 
    Jacobian matrix $\bj_{\ba}$ is invertible if and only if 
    $y_1,\ldots,y_\ell$ are distinct and nonzero. Together with 
    the Inverse Function Theorem, this implies that $\theta_\ba$ 
    is a local homeomorphism at $\by \in \mathbb{C}^\ell - \mathcal{A}_\ell$ 
    if the entries of $\by$ are distinct and nonzero. By the
    definition of $\phi$, we conclude that $\theta_\lambda$ is a 
    local homeomorphism at $\bz\in \C^\lambda$ if the 
    entries of $\bz$ are nonzero.
\end{proof}

\begin{rem}[Lifting critical value motions]
    The explicit local homeomorphism property described in 
    Theorem~\ref{thm:stratification} has an interesting 
    consequence. One can show that, given a specific complex 
    polynomial with distinct roots and a motion of its critical 
    values, there is a unique lift of this motion to a
    subspace of polynomials with critical points that continue to be 
    partitioned in the same fixed manner. This application is among 
    the primary motivations for the current work, and an explicit 
    statement will be stated and proved in a future article.
\end{rem}

%%%%%%%%%%%%%%%%%%%%%%%%%%%%

\bibliographystyle{amsalpha}
\providecommand{\bysame}{\leavevmode\hbox to3em{\hrulefill}\thinspace}
\providecommand{\MR}{\relax\ifhmode\unskip\space\fi MR }
\providecommand{\MRhref}[2]{%
  \href{http://www.ams.org/mathscinet-getitem?mr=#1}{#2}
}
\providecommand{\href}[2]{#2}

\end{document}